\documentclass[12pt]{article}

\usepackage{amsfonts}
\usepackage{amsmath}
\usepackage{amsthm}
\usepackage{amssymb}
\usepackage{graphicx}

\def\be{\begin{equation}}
\def\ee{\end{equation}}
\def\l{\langle}
\def\r{\rangle}
\def\p{\parallel}
\def\R{I\!\!R}
\def\cR{{\cal R}}
\def\cN{{\cal N}}
\def\ba{\bar{A}}
\newtheorem{remark}{Remark}
\newtheorem{lemma}{Lemma}

\newtheorem{proposition}{Proposition}

\begin{document}

\centerline{\large \bf 
The augmented Block Cimmino algorithm revisited}

\vspace*{0.3cm}
\centerline
{A. Dumitra\c sc$^+$, Ph. Leleux$^{++}$, C. Popa$^{+, \dag}$, D. Ruiz$^{++}$ and S. Torun$^{++}$ }

\begin{center}
{\small 
$^+$``Ovidius'' University of Constanta, Romania \\
$^{++}$CERFACS Toulouse, France\\
$^{\dag}$Institute of Statistical Mathematics and Applied Mathematics of the Romanian Academy and 
Academy of Romanian Scientists,  Bucharest, Romania
}
\end{center}

{\bf Abstract.}  In this paper we replay the definitions, constructions and results from \cite{duffsisc} by completing  and developing  some of them to inconsistent least squares problems.

\vspace*{0.3cm}

\vspace*{0.3cm}

\section*{Introduction}
In the paper \cite{duffsisc} the authors design and study a novel way of improving the computational efficiency of the block Cimmino method for consistent large sparse linear systems of equations. They propose an augmenting procedure of the system matrix, such  that the subspaces corresponding to the partitions are orthogonal. This results in a requirement to solve smaller linear  systems in parallel. In our paper, we replay and complete the analysis of this procedure, and  try an extension of it to inconsistent linear least squares problems.

We conclude this introductory part, by presenting the notations and definitions used in the rest of the paper.  By $\l \cdot, \cdot \r, \p \cdot \p$ we will denote the euclidean scalar product and norm on some space $\R^q$. If $A$ is a (real) $m \times n$ matrix we will denote by $A^T, a_i, a^j, rank(A), \cR(A),$ $\cN(A), A^+$ the transpose, $i$-th row, $j$-th column, rank, range, null space and Moore-Penrose pseudoinverse of it.
The vectors $x \in \R^q$ will be considered as column vectors, thus with the above rows and columns, the matrix $A$ can be written as
\be
\label{0-1}
A = \left[ 
		\begin{matrix}
		(a_1)^T \\
		(a_2)^T \\
		 \dots  \\
		(a_n)^T \\
		\end{matrix}
		\right] ~~\textrm{or}~~ A = [a^1 a^2 \dots a^n ].
\ee
If for $1 \leq p < m$, we split the rows indices as $1 \leq m_1 < m_2 < \dots < m_p = m$ and the subsets
\be
\label{0-2}
N_1 = \{ 1, \dots, m_1 \}, N_2 = \{ m_1+1, \dots, m_2 \}, \dots, 
N_p = \{ m_{p-1}+1, \dots, m_p \},
\ee
and define the row blocks $A_1, A_2, \dots, A_p$ of $A$ as 
\be
\label{0-33p}
A_1 = \left[ 
		\begin{matrix}
		(a_1)^T \\
			 \dots  \\
		(a_{m_1})^T \\
		\end{matrix}
		\right], ~~A_2 = \left[ 
		\begin{matrix}
		(a_{m_1+1})^T \\		 
		\dots  \\
		(a_{m_2})^T \\
		\end{matrix}
		\right], ~~A_p = \left[ 
		\begin{matrix}
		(a_{m_{p-1}+1})^T \\		 
		\dots  \\
		(a_{m_p})^T \\
		\end{matrix}
		\right],
\ee
then $A$ and $A^T$ will be written as 
\be
\label{0-4p}
 A = \left[ 
		\begin{matrix}
		A_1  \\
		\dots \\
		A_p
		\end{matrix}
		\right], ~~~A^T = [(A_1)^T (A_2)^T \dots (A_p)^T ].
\ee
 If $m=n$ and $A$ is invertible, $A^{-1}$  will denote its inverse. The orthogonal projector onto a vector subspace $S \in \R^q$ will be written as $P_S$ and the dimension of $S$ as $dim(S)$. $I_q, ~O_q$ will stand for the unit, respectively zero matrix of order $q$, and $D=diag(\delta_1, \delta_2, \dots, \delta_n)$ will be a notation for the diagonal matrix
$$
\left[ 
		\begin{matrix}
		\delta_1 & 0 & \dots & 0\\
		0 & \delta_2 & \dots & 0 \\
		 & & \ddots & \\
		0 & 0 &\dots &\delta_n \\
		\end{matrix}
		\right]
$$
The notations 
If $b \in \cR(A)$ the (classical)  solutions set and the minimal norm one will be denoted by $S(A; b), ~x_{LS}$, whereas in the inconsistent case ($b \notin \cR(A)$) the (least squares)  solutions set and the minimal norm one will be denoted by $LSS(A; b)$ and $x_{LS}$ too. 

\section{Consistent systems}
\label{consist}

Let $\tilde{A}: m \times n$, $\tilde{b} \in \R^m$, the consistent system $ \tilde{A} \tilde{x} = \tilde{b}$, $r=rank(\tilde{A}) \leq \min \{ m, n \}$ and $P$, $Q$ permutation matrices  such that 
\be
\label{1}
A =  P \tilde{A} Q =  \left[ 
		\begin{matrix}
		A_1  \\
		\dots \\
		A_p
		\end{matrix}
		\right], ~b=P \tilde{b}, ~x = Q^T \tilde{x},
\ee
with $A_1, \dots, A_p$ row blocks as in (\ref{0-33p}) (
details on such kind of transformations will be given in section \ref{rcm}). We have the equivalencies 
$$
 \tilde{A} \tilde{x} = \tilde{b} \Leftrightarrow ( P \tilde{A} Q) (Q^T \tilde{x}) = P \tilde{b} \Leftrightarrow 
$$
\be
\label{2}
A x = b,
\ee
therefore
\be
\label{200}
b \in \cR(A) ~~\textrm{and}~~ S(\tilde{A}; \tilde{b}) = Q(S(A; b)),
\ee
thus we will concentrate in what follows on the numerical solution of the system (\ref{2}), instead of the initial one $\tilde{A} \tilde{x} = \tilde{b}$. In \cite{duffsisc} the authors consider  an extended matrix $\ba: m \times \bar{n}$ 
	\be
\label{3}
\ba = [ A ~~\Gamma ] = 
		 \left[ 
		\begin{matrix}
		\ba_1  \\
		\dots \\
		\ba_p
		\end{matrix}
		\right] = \left[ 
		\begin{matrix}
		A_1 ~~\Gamma_1  \\
		\dots \\
		A_p ~~\Gamma_p
		\end{matrix}
		\right], ~~\bar{n}=n+q, 
\ee
	with $\Gamma: m \times q$, such that the row blocks $\ba_i$ are mutually orthogonal, i.e. 	
	\be
	\label{202}
	\ba_i \ba^T_j =0, ~~\forall i \neq j. 
	\ee 
Details on such kind of extensions will be given in section \ref{rcm}. The extended system
\be
\label{201}
\left[ 
		\begin{matrix}
		\bar{A} \\
		Y \\
		\end{matrix}
		\right] = 
\left[ 
		\begin{matrix}
		A & \Gamma \\
		0 & I_q \\
		\end{matrix}
		\right] \left[ 
		\begin{matrix}
		x \\
		y \\ 
		\end{matrix}
		\right] = \left[ 
		\begin{matrix}
		b \\
		0 \\ 
		\end{matrix}
		\right]
\ee
 has the same set of solutions as (\ref{2}), but   the blocks $\bar{A} = [ A ~~\Gamma ]$ and $Y = [ 0 ~~I_q ]$ are no more orthogonal. In order to overcome this aspect the authors consider in \cite{duffsisc} instead of (\ref{201}) the system
\be
\label{7}
\left[ 
		\begin{matrix}
		A & \Gamma \\
		B & S \\
		\end{matrix}
		\right] \left[ 
		\begin{matrix}
		x \\
		y \\ 
		\end{matrix}
		\right] = \left[ 
		\begin{matrix}
		b \\
		f \\ 
		\end{matrix}
		\right] .
		\ee
	We will show that, for an appropriate choice of the $q \times \bar{n}$ matrix $W = [ B ~~S ]$ and $f \in \R^q$ the system (\ref{7}) provides all the solutions of $Ax=b$. In this respect we need the following result (for the proof see  \cite{punt}, page 121, eq. (5.2)). 
\begin{proposition}
\label{puntm}
(i) If $E, F$ are two matrices with the same number of rows (say $\tau$) then 
\be
\label{p1}
rank([E ~~F]) = rank(E) ~+~ rank((I_{\tau} - E E^+)F).
\ee
(ii) If $E, F$ are two matrices such that $E: \tau \times \mu$, $F: \mu \times \delta$ then
\be
\label{p11p}
rank(EF) = rank(E) - dim(\cR(E^T) \cap \cN(F^T)).
\ee
\end{proposition}
Let also 	$P = P_{\cR(\ba^T)}$, which under the mutual  orthogonality hypothesis (\ref{202}) is given by (see e.g.  \cite{horn})
	\be
\label{4}
P = \bar{A}^+ \bar{A} = \sum_{i=1}^p P_{\cR(\ba_i^T)} ~\textrm{with}~ 
 P_{\cR(\ba_i^T)} =  \ba^+_i \ba_i.
\ee
\begin{proposition}
\label{lem1}
(i) If
\be
\label{5}
 W = Y (I -P),
\ee
with $Y = [ 0 ~I_q ]$ (see (\ref{201})), 
then the row block $W = [ B ~~S ]$ is orthogonal on $\ba$, hence on each row block $\ba_i, i=1, \dots, p$.\\
(ii) We have the equalities
\be
\label{6}
\left[ 
		\begin{matrix}
		\ba \\
		W
		\end{matrix}
		\right]^+ = [ \ba^+  ~W^+ ],  ~\textrm{and}
		\ee
		\be
\label{6p}
		W W^T = B B^T + S^2 = S, ~\textrm{where}~  S = Y (I - P ) Y^T.
\ee
(iii) Let us suppose that the matrix $S$ from (\ref{6p}) is invertible and  $f$ is given by
\be
\label{5p}
f=-Y \ba^+ b.
\ee
Then, if $x$ a solution of (\ref{2}), the vector $\left[ 
		\begin{matrix}
		x \\
		0 \\ 
		\end{matrix}
		\right]$ is a solution of (\ref{7}). Conversely, if 
$\left[ 
		\begin{matrix}
		x \\
		y \\ 
		\end{matrix}
		\right]$ is the minimal norm solution of the system (\ref{7}), then $y=0$ and $x$ is a solution of (\ref{2}).
\end{proposition}
	\begin{proof}
	(i) We have from (\ref{4})-(\ref{5})
	\be
	\label{7p}
	\ba W^T = \ba [ B ~~S ]^T = \ba (I - P)^T Y^T = \ba P_{\cN(\ba^T)} Y^T = 0.
	\ee
	(ii) We show by simple computations, also using (\ref{7p}) that the matrices $\left[ 
		\begin{matrix}
		\ba \\
		W
		\end{matrix}
		\right]$ and $[ \ba^+  ~W^+ ]$ satisfy the four Penrose equalities, which uniquely characterize the Moore-Penrose pseudoinverse (see e.g. \cite{horn}). The  equalities in (\ref{6p}) are proved in \cite{duffsisc}.\\
(iii)  If $x$ is a solution of (\ref{2}), then $Ax=b$ and we have (see (\ref{7}), (\ref{6p}), (\ref{5p}) and  \cite{duffsisc})
$$
\left[ 
		\begin{matrix}
		A & \Gamma \\
		B & S \\
		\end{matrix}
		\right] \left[ 
		\begin{matrix}
		x \\
		0 \\ 
		\end{matrix}
		\right] = \left[ 
		\begin{matrix}
		Ax \\
		W \left[ 
		\begin{matrix}
		x \\
		0 \\ 
		\end{matrix}
		\right] \\ 
		\end{matrix}
		\right] = \left[ 
		\begin{matrix}
		b \\
		Y(I-P) \left[ 
		\begin{matrix}
		x \\
		0 \\ 
		\end{matrix}
		\right] \\ 
		\end{matrix}
		\right] = $$
		$$
		\left[ 
		\begin{matrix}
		b \\
		-Y \bar{A}^+ \bar{A} \left[ 
		\begin{matrix}
		x \\
		0 \\ 
		\end{matrix}
		\right] \\ 
		\end{matrix}
		\right] = \left[ 
		\begin{matrix}
		b \\
		-Y \bar{A}^+ b \\ 
		\end{matrix}
		\right] = \left[ 
		\begin{matrix}
		b\\
		f \\ 
		\end{matrix}
		\right],
$$
with $f$ from (\ref{5p}), which completes the first part of the proof. \\Let now $\left[ 
		\begin{matrix}
		x \\
		y
		\end{matrix}
		\right]$ be the minimal norm solution of (\ref{7}) with $f$ from (\ref{5p}).  Hence (see e.g. \cite{popabook} and (\ref{6})) 
			\be
	\label{8}
		\left[ 
		\begin{matrix}
		x \\
		y
		\end{matrix}
		\right] = \left[ 
		\begin{matrix}
		\ba \\
		W
		\end{matrix}
		\right]^+ \left[ 
		\begin{matrix}
		b \\
		f
		\end{matrix}
		\right]  = [ \ba^+  ~W^+ ] \left[ 
		\begin{matrix}
		b \\
		f
		\end{matrix}
		\right] =  \ba^+ b + W^+ f.
		\ee
	According to our hypothesis on the matrix $S$ it results that $W^T$ has full column rank, therefore (see again \cite{popabook}) 
	$$
	W^+ = W^T (W W^T)^{-1} = (I-P) Y^T S^{-1} = W^T S^{-1}.
	$$
	Then
	$$
	W^+ f = W^T S^{-1} f = \left[ 
		\begin{matrix}
		B^T \\
		S^T \\ 
		\end{matrix}
		\right] S^{-1} f = \left[ 
		\begin{matrix}
		B^T S^{-1} f\\
		f \\ 
		\end{matrix}
		\right],
	$$
	and from (\ref{8}) and (\ref{5}) we obtain ($\bar{n} = n + q$;  see (\ref{3}))
	$$
	\ba^+ b + W^+ f = \left[
	\begin{matrix}
	I_n & 0\\
	0 & I_q \\
	\end{matrix}
	\right ]
	\ba^+ b  + W^+ f= 
	\left[ 
		\begin{matrix}
		[I_n ~0] \ba^+ b\\
		[0 ~I_q] \ba^+ b\\ 
		\end{matrix}
		\right] + \left[ 
		\begin{matrix}
		B^T S^{-1} f\\
		f \\ 
		\end{matrix}
		\right]   = 
	$$
	\be
	\label{9}
\left[ 
		\begin{matrix}
		[I_n ~0] \ba^+ b\\
		Y \ba^+ b \\ 
		\end{matrix}
		\right] + \left[ 
		\begin{matrix}
		B^T S^{-1} f\\
		-Y \ba^+ b \\ 
		\end{matrix}
		\right] = \left[ 
		\begin{matrix}
		[I_n ~0] \ba^+ b + B^T S^{-1} f\\
		0 \\ 
		\end{matrix}
		\right] .
	\ee
	The equalities (\ref{8}) and (\ref{9}) give us $y=0$,  hence 
	 the minimal norm solution of the (consistent) system (\ref{7}) has the form $\left[ 
		\begin{matrix}
		x \\
		0 \\ 
		\end{matrix}
		\right]$. In particular we have 
	$$\left[ 
		\begin{matrix}
		b \\
		f \\ 
		\end{matrix}
		\right] = \left[ 
		\begin{matrix}
		A & \Gamma \\
		B & S \\
		\end{matrix}
		\right] \left[ 
		\begin{matrix}
		x \\
		0 \\ 
		\end{matrix}
		\right] = \left[ 
		\begin{matrix}
		A x \\
		B x \\ 
		\end{matrix}
		\right],
	$$
	i.e. $b = Ax$ which tell us that $x$ is a solution of the system (\ref{2}) and completes the proof.
	\end{proof}
Concerning the assumption on the invertibility of the matrix $S$ from (\ref{6p}) we have the following result. 
	\begin{lemma}
\label{lem1p}
If $m \leq n$ and the matrix $A$ from (\ref{1}) has full row rank, then $S$ is invertible.
\end{lemma}
\begin{proof}
\textbf{Version 1.} According to the equality $S=W W^T$ (see (\ref{6p})) we get invertibility for $S: q \times q$ if the matrix $W^T: (n+q) \times  q$ (see (\ref{3}))  has full column rank. In this respect, let us suppose that $W^T z = 0$, for some $z \in \R^q$. But, because $A$ has full row rank also the matrix $\bar{A}$ will have full row rank, then (see e.g \cite{popabook})   
\be
\label{100}
\bar{A}^+ = \bar{A}^T (\bar{A} \bar{A}^T)^{-1}, ~~P= \bar{A}^+ \bar{A} = \bar{A}^T (\bar{A} \bar{A}^T)^{-1} \bar{A}.
\ee
Therefore, from (\ref{3}), (\ref{5}), (\ref{100}) and (\ref{4}) we obtain
$$
W^T z =0 \Leftrightarrow (I-P) Y^T z=0 \Leftrightarrow (I-P) \left[ 
		\begin{matrix}
		0 \\
		z \\ 
		\end{matrix}
		\right] = 0 \Leftrightarrow 
$$
$$
\left[ 
		\begin{matrix}
		0 \\
		z \\ 
		\end{matrix}
		\right] - \left[ 
		\begin{matrix}
		A^T \\
		\Gamma^T \\ 
		\end{matrix}
		\right] (A A^T + \Gamma \Gamma^T)^{-1} [ A ~~\Gamma ] \left[ 
		\begin{matrix}
		0 \\
		z \\ 
		\end{matrix}
		\right] = \left[ 
		\begin{matrix}
		0 \\
		0 \\ 
		\end{matrix}
		\right] \Leftrightarrow
$$
\be
\label{101}
 \left\{\begin{array}{c}
A^T (A A^T + \Gamma \Gamma^T)^{-1} \Gamma z = 0 \\
z - \Gamma^T (A A^T + \Gamma \Gamma^T)^{-1} \Gamma z = 0 \\
\end{array}\right.
\ee
But, from our hypothesis the matrix $A^T$ has full column rank, thus from the first equation in (\ref{101}) we get $\Gamma z = 0$, which gives us $z=0$ from the second equation.\\
\textbf{Version 2.} We will apply the result in (\ref{p11p}) for $E=Y: q \times \bar{n}$, $F=I-P: \bar{n} \times \bar{n}$ by first observing that
$$
\cN(F^T) = \cN(I-P)= \{ x \in \R^{\bar{n}}, Px=x \} = \{  x \in \R^{\bar{n}}, P_{\cR(\bar{A}^T)}(x) = x \} = \cR(\bar{A}^T),
$$
which gives us in  (\ref{p11p}) (see also (\ref{5})) 
$$
rank(W)=rank(Y(I-P))= rank(Y)-dim(\cR(Y^T) \cap \cR(\bar{A}^T)) = 
$$
\be
\label{p11bis}
q - dim(\cR(Y^T) \cap \cR(\bar{A}^T)).
\ee
Let now $y \in \R^{\bar{n}}$ be such that (see also (\ref{3}))
$$
y \in \cR(Y^T) \cap \cR(\bar{A}^T) = \cR( \left[ 
		\begin{matrix}
		0 \\
		I_q \\ 
		\end{matrix}
		\right] ) \cap \cR( \left[ 
		\begin{matrix}
		A^T \\
		\Gamma^T \\ 
		\end{matrix}
		\right] ).
$$
Then, for some $x \in \R^q,$ $y =  \left[ 
		\begin{matrix}
		0 \\
		x \\ 
		\end{matrix}
		\right] \in \cR( \left[ 
		\begin{matrix}
		A^T \\
		\Gamma^T \\ 
		\end{matrix}
		\right] )$, i.e. it exists $z \in \R^m$ such that
		\be
		\label{221p}
		\left[ 
		\begin{matrix}
		0 \\
		x \\ 
		\end{matrix}
		\right] =  \left[ 
		\begin{matrix}
		A^T z \\
		\Gamma^T z\\ 
		\end{matrix}
		\right] \Leftrightarrow A^T z=0, ~\Gamma^T z =x.
		\ee
		But, as $A$ has full row rank, $A^T$ has full column rank, therefore from $A^T z = 0$ we get $z=0$, thus $y=0$, i.e.
		\be
		\label{222p}
		dim(\cR(Y^T) \cap \cR(\bar{A}^T)) = 0,
		\ee
		which together with (\ref{p11bis}) gives us $rank(W) = q$. But, from (\ref{6p}) we have that $rank(S)=rank(W W^T) = rank(W)$, i.e. $S$ is invertible.
\end{proof} 
	\begin{remark}
	\label{rrem}
	(i) Version 2 of the above proof gives us the possibility to formulate the following conjecture:
	
	\vspace*{0.2cm}
	\textbf{If $1 \leq rank(\bar{A}) = r < m$, it exists $y \in \cR(Y^T) \cap \cR(\bar{A}^T)$, $y \neq 0$.} 
	
		\vspace*{0.2cm}
	This would give us that $dim(\cR(Y^T) \cap \cR(\bar{A}^T)) \geq 1$, therefore 
	$rank(W) = rank(S) \leq q-1$, i.e. the matrix $S$ is no more invertible.\\
(ii) Unfortunately, we  do not  have a proof of the fact that the  converse is  true (e.g. that $S$ is not invertible iff $\bar{A}$ does not have full row rank).\\ 
(iii) Moreover, because the matrices $Y=[0 ~~I_q]$ and $\bar{A} = [A ~~\Gamma]$ are connected only by the dimension $q$ of the block $I_q$ (which corresponds to the number of columns in the block $\Gamma$), we have to consider the special constructions of $\Gamma$ from section \ref{rcm} (see also \cite{duffsisc}) which give us additional information.
		\end{remark}
\section{Inconsistent systems}
\label{sec2}	
	If the initial system  $\tilde{A} \tilde{x} = \tilde{b}$ is no more consistent it must be reformulated in the least squares sense as 
\be
\label{205}
\p \tilde{A} \tilde{x} - \tilde{b} \p = \min !
\ee 
Moreover, because the matrices $P$ and $Q$ in (\ref{1}) are orthogonal we have the equivalences 
$$
\p \tilde{A} \tilde{x} - \tilde{b} \p = \min !  ~\Leftrightarrow~ 
\p ( P \tilde{A} Q) (Q^T \tilde{x}) - P \tilde{b} \p = \min !  ~\Leftrightarrow~ 
$$
\be
\label{2p}
\p A x - b \p =\min ! 
\ee
which give us (see also \cite{popabook})
\be
\label{206}
LSS(\tilde{A}; \tilde{b}) = Q(LSS(A; b)).
\ee
Unfortunately, in this case we cannot adapt the results from section \ref{consist} because the extended system (\ref{201}) must be reformulated in a least squares sense which is no more equivalent with the problems in (\ref{2p}). For this reason we must use one of the  two consistent (sparse) equivalent formulations, namely: 
\begin{itemize}
\item[(i)] \textbf{Augmented system} (see e.g. \cite{mario})
\be
\label{0-4}
\left[ 
		\begin{matrix}
		I & \tilde{A} \\
		\tilde{A}^T & 0 \\
		\end{matrix}
		\right] \left[ 
		\begin{matrix}
		\tilde{r} \\
		\tilde{x} \\ 
		\end{matrix}
		\right] = \left[ 
		\begin{matrix}
		\tilde{b} \\
		0 \\ 
		\end{matrix}
		\right] ~~{\widehat{\Leftrightarrow}}~~  
		\left[ 
		\begin{matrix}
		I & {A} \\
		{A}^T & 0 \\
		\end{matrix}
		\right] \left[ 
		\begin{matrix}
		r \\
		x \\ 
		\end{matrix}
		\right] = \left[ 
		\begin{matrix}
		{b} \\
		0 \\ 
		\end{matrix}
		\right] .
\ee
\begin{remark}
\label{may16}
The above equivalence ${\widehat{\Leftrightarrow}}$ has to be understood in the following sense.
\begin{itemize}
\item[1.]  If $[ \tilde{r} ~\tilde{x} ]^T$ is a solution of the left system in (\ref{0-4}), we know that $\tilde{A}^T \tilde{r} = 0$ and $\tilde{A}^T \tilde{A} \tilde{x} = \tilde{A}^T \tilde{b}$, then $\tilde{x} \in LSS(\tilde{A}; \tilde{b})$, i.e. (see (\ref{206})) $x = Q^T \tilde{x} \in LSS(A; b)$; then $A^T A x = A^T b$, and if we define $r=b-Ax$ it results that $[r ~x ]^T$ is a solution of the right system in (\ref{0-4}).

\vspace*{0.2cm}
\item[2.]  If $[ {r} ~{x} ]^T$ is a solution of the right  system in (\ref{0-4}), we know that ${A}^T {r} = 0$ and ${A}^T {A} {x} = {A}^T {b}$, then ${x} \in LSS({A}; {b})$, i.e. (see (\ref{206})) $\tilde{x} = Q {x} \in LSS(\tilde{A}; \tilde{b})$; then $\tilde{A}^T \tilde{A} \tilde{x} = \tilde{A}^T \tilde{b}$, and if we define $\tilde{r}=\tilde{b} - \tilde{A} \tilde{x}$ it results that $[\tilde{r} ~\tilde{x} ]^T$ is a solution of the left  system in (\ref{0-4}).
\end{itemize}
\end{remark}
\item[(ii)] \textbf{Consistent right hand side} (see e.g. \cite{popabook}, (\ref{206}) and Remark \ref{may16} above)
\be
\label{0-3}
\tilde{A} \tilde{x} = \tilde{b}_{\tilde{A}} ~\Leftrightarrow~ 
\tilde{x} \in LSS(\tilde{A}; \tilde{b}) ~\Leftrightarrow~ 
x = Q^T \tilde{x} \in LSS(A; b) ~\Leftrightarrow~ 
Ax = {b}_{{A}},
\ee
where
\be
\label{0-3p}
\tilde{b}_{\tilde{A}} = P_{\cR(\tilde{A})}(\tilde{b}), ~~b_A =  P_{\cR(A)}(b).
\ee
\end{itemize} 
The method from section \ref{consist} can be directly adapted for computing the minimal norm solution $x$ for the second system in  (\ref{0-4}), and then constructing a solution $\tilde{x}$ through the equality $\tilde{x} = Qx$ (see (\ref{1}) and section \ref{rcm} for details related to the application of Cuthill-McKee (CM) algorithm to the matrix from (\ref{0-4})).   \\
Concerning the second approach, we first need to compute
 (see e.g. \cite{popabook}) 
\be
\label{205p}
b_A = P_{{A}}({b}) = A {A}^+ {b},
\ee
then solve the consistent system $Ax=b_A$ and get the minimal norm solution $x$ and then construct a solution $\tilde{x}$ through the equality $\tilde{x} = Qx$ (see (\ref{1})). \\
Computing $b_A$ is a difficult task, work is in progress  on this subject and it will be presented in a future paper. In the present one we will only observe that $P_{{\bar{A}}}({b})$ can be computed in parallel, using the same block structure  (\ref{3}) of the matrix $\bar{A}$, and  the orthogonality property (\ref{4}).
\begin{proposition}
\label{lem2}
(i) Let $\ba_i$ be the blocks from (\ref{3}) and 
\be
\label{10}
B_i = \ba^T_i, ~~B_i^T B_j =0, i\neq j, ~~B=[B_1, B_2, \dots, B_p].
\ee
Then
\be
\label{11}
P_{\cR(B)} = \sum_{i=1}^p P_{\cR(B_i)}, ~~B_i^+ B_j = 0, ~~ B_i B^+_j =0, ~~\forall i \neq j.
\ee
(ii) Let $\ba_i: r_i \times \bar{n}, r_1 + \dots r_p = m$, and for $z \in \R^m$ we denote by $z^i \in \R^{r_i}$ the corresponding subvector, i.e.
$$
z = \left[ 
		\begin{matrix}
		z^1 \\
		\dots \\
		z^p \\ 
		\end{matrix}
		\right] \in \R^{r_1} \times \dots \times \R^{r_p}.
$$
Then
\be
\label{12}
P_{\cR(B^T)}(z) = P_{\cR(\ba)}(z) = \left[ 
		\begin{matrix}
		B^+_1 B_1 z^1 \\
		\dots \\
		B^+_p B_p z^p \\ 
		\end{matrix}
		\right] .
\ee
\end{proposition}
	\begin{proof}
	(i) The first equality in (\ref{11}) results from the definition (\ref{10}) of the blocks $B_i$ and the first equality in (\ref{4}) (also observe that from (\ref{10}) and (\ref{3}) we obtain $B=A^T$). 
	 We will prove only the second equality in (\ref{11}); for the third one, similar arguments are available. Let 
	$$U^T B_i V = \Sigma = diag(\sigma_1, \dots, \sigma_r) $$
	be a singular value decomposition of $B_i$. Then 
	\be
	\label{15}
	B_i^T = V \Sigma^T U^T, ~~B^+_i = V \Sigma^+ U^T.
	\ee
	Therefore, from the hypothesis $B^T_i B_j =0$ it results 
	$V \Sigma^T U^T B_j = 0$, hence (because $V$ is invertible) 
	$\Sigma^T (U^T B_j) = 0.$ But, because the matrices $\Sigma^T$ and $\Sigma^+$ have the same dimensions and structure, we also have 
	$\Sigma^+ (U^T B_j) = 0$, which according to (\ref{15}) gives us $B^+_i Bj =0$.
	\\
	 (ii) We know that (see e.g. \cite{horn})
	 \be
	 \label{13}
	 P_{\cR(\ba)} = P_{\cR(B^T)} = B^+ B.
	 \ee
	But, from the mutual orthogonality of the blocks $B_i, i=1, \dots, p$ we obtain by direct computation that
	 \be
	 \label{14}
	 B^+ = \left[ 
		\begin{matrix}
		B^+_1  \\
		\dots \\
		B^+_p  \\ 
		\end{matrix}
		\right] .
	 \ee
	Then, according to (\ref{13}) - (\ref{14}) and the formula for the product $B^+ B$  we get
	$
	P_{\cR(B^T)}(z) = 	P_{\cR(\bar{A})}(z) = B^+ B z = $
	$$
	\left[ 
		\begin{matrix}
		B^+_1 B_1 & 0 & \dots & 0\\
		0 & B^+_2 B_2 & \dots & 0 \\
		\dots & \dots & \dots & \dots \\
	0 & 0 & \dots &	B^+_p B_p  \\ 
		\end{matrix}
		\right] \left[ 
		\begin{matrix}
		z^1  \\
		z^2 \\
		\dots \\
		z^p  \\ 
		\end{matrix}
		\right] = \left[ 
		\begin{matrix}
		B^+_1 B_1 z^1 \\
		\dots \\
		B^+_p B_p z^p \\ 
		\end{matrix}
		\right]
	$$
			and the proof is complete.
	\end{proof}
\section{Constructing the matrix $A$ from (\ref{1})}
\label{rcm}

As we already mentioned in section \ref{consist} we will present in this section some procedures to construct a row block splitting of $\tilde{A}$ as in (\ref{1}). This also applies (in the inconsistent case  (\ref{2p})) for the augmented system equivalent formulation (\ref{0-4}) right. 
	Let $\tilde{A}$ be an ${m \times n}$  sparse rectangular matrix. We can attach to it the bipartite graph $G=(E,R,C)$, with $R=\left\lbrace 1,2,...,m \right\rbrace $ the set of nodes denoting row indices, $C=\left\lbrace 1,2,...,n \right\rbrace $ the set of nodes denoting column indices and 
    
    \begin{equation*}
         E=\left\lbrace (i,j) \ | \ \tilde{A}_{i,j} \neq 0, \ i \in R, \ j \in C \right\rbrace 
    \end{equation*}
     the set of edges. The bipartite graph associated to $ A $ can be seen in Fig. \eqref{fig:bipartite}  We note that the following considerations assume the graph is connected. If that is not the case, the procedure can be applied to every connected component of the graph.
     
	\begin{figure}[h]
        \centering
         \includegraphics[width=.3\textwidth]{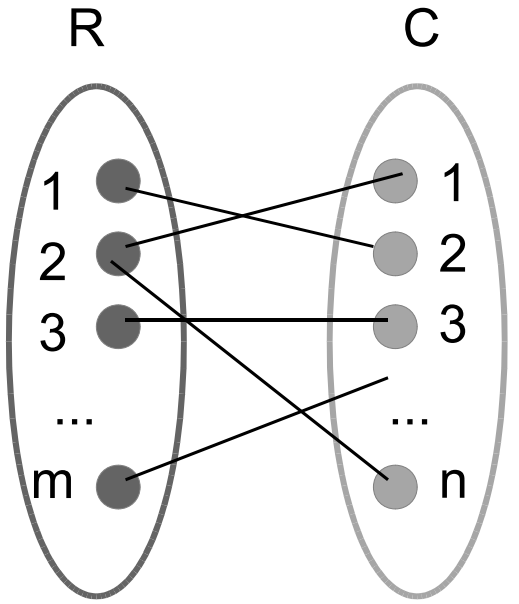}
         \caption{Bipartite graph}
         \label{fig:bipartite}
     \end{figure}
		
	The adjacency relation in the bipartite graph can be defined as follows:
    
	\begin{equation} \label{bip1}
	\nexists \ (i,j) \ | \ i,j \in R, \quad	\nexists \ (i,j) \ | \ i,j \in C,
	\end{equation}
	
	\begin{equation} \label{bip2}
	\exists \ (i,j) \ | \ i \in R, \ j \in C \Leftrightarrow \exists \ (j,i) \ | \ i \in R, \ j \in C.
	\end{equation}
	
	\paragraph{}
	The same information can be represented by the graph $\mathcal{G}=(\mathcal{E}, \mathcal{V})$ with $ \mathcal{V}=\left\lbrace  1,2,...,m,m+1,m+2,...,m+n \right\rbrace $, the set of nodes denoting both row and column indices, i.e. $\mathcal{V}=R \cup \overline{C}, \ \overline{C}=\left\lbrace c+m, \ c \in C \right\rbrace  $  and 
    
    \begin{equation*}  
        \mathcal{E}=\left\lbrace (i,\overline{j}) \ |  \ , \ i \in R, \overline{j}=j+m, \ j \in C, \  \tilde{A}_{i,j} \neq 0 \right\rbrace 
     \end{equation*}
     the set of edges. The adjacency matrix for $\mathcal{G}$ is the symmetric matrix (see \eqref{bip2})
	
	\begin{equation}
	\label{211}
		\hat{A}=
		 \left[ 
		\begin{matrix}
		0 & \tilde{A} \\
		\tilde{A}^T & 0 
		\end{matrix}
		\right] .  
	\end{equation}	
	Due to the fact that there is no row-row or column-column adjacency (see \eqref{bip1}),  $\hat{A}=$  contains two  zero  blocks on its diagonal, with sizes $m \times m$ and $n \times n$ respectively. We seek to reorder $A$ such that its nonzeros are closer to the diagonal. To this end, we apply the Cuthill-McKee (CM) algorithm (see \cite [Chapter~8]{duff2017direct} ) to $\hat{A}$ and use the result to reorder $A$, thus reducing its bandwidth, and giving it a special structure.	The CM algorithm can be thought of as a particular form of Breadth First Search (BFS), where the neighbors of each node are visited with respect to an increasing order of their degree, and the starting node is the one having the minimum degree.	Visiting a node's neighbors implies relabeling them with the smallest unused label. For example, given a node $i$ with neighbors $p, \ t, \ q$, CM will sort them in increasing order of degree, labeling them as $i+1, \ i+2, \ i+3$, if these labels are still available. 	A level set is the set of nodes (not yet labeled) neighboring at least one node of the previous level set. The first level set contains only the starting node. The particular form of adjacency described above (see \eqref{bip1}) also implies that level sets alternate between sets of row indices and sets of column indices.
	
	\begin{equation*}
		s:  \text{ starting node},\quad 	S_1=\{ s \},
	\end{equation*}
	
	\begin{equation*}
		S_i=\{ j \ | \ \exists \ (k,j) \in \mathcal{E}, \ k \in S_{i-1}, \ j \in \overline{\mathcal{V}}  \ \backslash \ \stackrel{i-1}{ \underset{p=1}{\bigcup} } S_p \},
	\end{equation*}
	
	\begin{equation}\label{set1}
		\text{If }  S_i=\{ j \ | \ j \le m \} \text{ then } S_{i+1}=\{ j \ | \ j > m \} ,
	\end{equation}
	
	\begin{equation}\label{set2}
		\text{If }  S_i=\{ j \ | \ j > m \} \text{ then } S_{i+1}=\{ j \ | \ j \le m \}.
	\end{equation}
	The reordered matrix $\hat{A}$ is $\hat{A}^R$, with every row containing the relabeled neighbors of the node on the diagonal, where $A_{st}$ is the submatrix of $A$ corresponding to the rows of row level set $s$ and columns of column level set $t$. Every level  set is represented by a square block on the diagonal, of size equal to the number of nodes in the set. Due to adjacency, these blocks are  zero blocks, and the matrix is symmetric.
	\begin{equation}
	\label{212}
    \hat{A}^R=
	\left[ 
	\begin{matrix}
           0 & A_{11}  \\
	A_{11}^T &     0    & A_{21}^T \\
             & A_{21}   &     0     & A_{22} \\
             &          &  A_{22}^T &     0   & A_{32}^T \\
             &          &           &  A_{32} &     0     & A_{33} \\
             &          &           &         &  A_{33}^T &     0   & A_{43}^T \\
             &          &           &         &           &  ...    &     ...   & ... \\
	\end{matrix}
	\right] .
	\end{equation}	
	The blocks above the diagonal refer the new, previously unlabeled nodes, which will constitute the next level set. Thus their size is $n_i \times n_{i+1}$, the number of nodes in the current level set, and that of the next, respectively. If the starting node of the CM reordering corresponds to a row index, odd level sets will contain row indices while even level sets will contain column indices (see \eqref{set1}) and \eqref{set2}.	The graph traversal by CM ensures that all nodes are visited, thus:
	\begin{equation} \label{card}
		\left| \underset{k=0}{\bigcup} S_{2k+1} \right| =m, \quad \left| \underset{k=1}{\bigcup} S_{2k} \right| =n.
	\end{equation}
	Using the information related to the construction of $\hat{A}^R$ we can produce a reordering of rows and columns of the original matrix $A$ as described below (see \cite{reid2006reducing}, page 809): {\it ``if we permute the rows of $A$ by the row level sets and the orderings within them, and permute the columns by the column level sets and the orderings within them, we find the block bidiagonal matrix $A^R$ written as''}
       	\begin{equation}\label{AR}
    A^R=
    \left[ 
    \begin{matrix}
    A_{11}  \\
    A_{21}      &  A_{22} \\
                &  A_{32} & A_{33} \\
                &         & A_{43} & A_{44} \\
                &         &        & ...      & ... \\
    \end{matrix}
    \right]   
    \begin{matrix}
    \}\ n_1 \\
    \}\ n_3 \\
    \}\ n_5 \\
    \}\ n_7 \\
    ...
    \end{matrix}
    \end{equation}
        \begin{equation*}
    \ \ \ n_2 \quad \ n_4 \quad \  n_6 \ \quad  n_8  \ \quad ...
    \end{equation*}     
    with $n_i= |S_i|$, the sizes of the blocks. Similar arguments can be made if the first node chosen by CM represents a column index. In this case, the odd and even sets have swapped contents, with statements \eqref{card} and \eqref{AR} changed accordingly. Additionally, $A^R$ is built using the transposed blocks of $\hat{A}^R$.    
   	The CM algorithm outputs a set $ \overline{\mathcal{V}}$, a reordering of $\mathcal{V}$ based on the successive concatenation of the level sets. Let $P$ be the matrix obtained by permuting the rows of $I_m $ such that their order is the same as that of the row indices in $ \overline{\mathcal{V}}$.  Let $Q$ be the matrix obtained by permuting the columns of $I_n$ to mirror the order of column indices in $ \overline{\mathcal{V}}$. The matrix $A^R$ is a permutation of $A$, considering that the blocks in \eqref{AR} contain all the nonzeros of $A$ and $A^R$ has the same size as  $A$ (see (\ref{card})). As described above (see \cite{reid2006reducing}, page 809), we find permutation matrices $P$ and $Q$ such that $PAQ$ yields this permutation. Therefore, we have 
    \begin{equation}\label{paq}
    A^R=P \tilde{A} Q.
    \end{equation}
 In the inconsistent case of the system $\tilde{A} \tilde{x} = \tilde{b}$ we use the sparse equivalent augmented system 
\be
\label{210}
\left[ 
		\begin{matrix}
		I & \tilde{A} \\
		\tilde{A}^T & 0 \\
		\end{matrix}
		\right] \left[ 
		\begin{matrix}
		\tilde{r} \\
		\tilde{x} \\ 
		\end{matrix}
		\right] = \left[ 
		\begin{matrix}
		\tilde{b} \\
		0 \\ 
		\end{matrix}
		\right] ~\Leftrightarrow~ \p \tilde{A} \tilde{x} - \tilde{b} \p = \min!
\ee
   We then apply the CM algorithm to the (symmetric) matrix 
	\be
	\label{210p}
	\left[ 
		\begin{matrix}
		I & \tilde{A} \\
		\tilde{A}^T & 0 \\
		\end{matrix}
		\right]
	\ee
		in (\ref{210}) and get a matrix $A^R$ of the form  (see (\ref{212}))
		\begin{equation}
	\label{213}
    \hat{A}^R=
	\left[ 
	\begin{matrix}
           E_{11} & A_{11}  \\
	A_{11}^T &     E_{22}    & A_{21}^T \\
             & A_{21}   &     E_{33}     & A_{22} \\
             &          &  A_{22}^T &     E_{44}   & A_{32}^T \\
             &          &           &  A_{32} &     E_{55}     & A_{33} \\
             &          &           &         &  A_{33}^T &     E_{66}   & A_{43}^T \\
             &          &           &         &           &  ...    &     ...   & ... \\
	\end{matrix}
	\right], 
	\end{equation}
	where $E_{ii} = I_{n_i}$ or $E_{ii} = O_{n_i}$ depending on  the starting index of the CM algorithm. The $ I $ and $ O $ blocks alternate starting with $ I $ if the first index chosen by CM corresponds to a row or $ O $ if it is a column index.
	Therefore, according to the procedure of the CM  algorithm we find permutation matrices $\hat{P}, \hat{Q}$ such that
		\be
		\label{214}
		   \hat{A}^R=\hat{P}  ~\hat{A} ~\hat{Q},
		\ee
		hence the equivalent consistent systems
	\be
	\label{214p}
	\hat{A} \left[ 
		\begin{matrix}
		\tilde{r} \\
		\tilde{x} \\
		\end{matrix}
		\right] = \left[ 
		\begin{matrix}
		\tilde{b} \\
		0 \\
		\end{matrix}
		\right] ~\Leftrightarrow~ (\hat{P} \hat{A} \hat{Q}) (\hat{Q}^T \left[ 
		\begin{matrix}
		\tilde{r} \\
		\tilde{x} \\
		\end{matrix}
		\right]) = \hat{P} \left[ 
		\begin{matrix}
		\tilde{b} \\
		0 \\
		\end{matrix}
		\right] ~\Leftrightarrow~ 
		\hat{A}^R \left[ 
		\begin{matrix}
		\hat{r} \\
		\hat{x} \\
		\end{matrix}
		\right] = \left[ 
		\begin{matrix}
		\hat{b} \\
		\hat{c} \\
		\end{matrix}
		\right],
	\ee
	with $\hat{A}^R$ from (\ref{214}) and 
	\be
	\label{215}
	\left[ 
		\begin{matrix}
		\hat{r} \\
		\hat{x} \\
		\end{matrix}
		\right] = \hat{Q}^T \left[ 
		\begin{matrix}
		\tilde{r} \\
		\tilde{x} \\
		\end{matrix}
		\right], ~~~~\left[ 
		\begin{matrix}
		\hat{b} \\
		\hat{c} \\
		\end{matrix}
		\right] = \hat{P} \left[ 
		\begin{matrix}
		\tilde{b} \\
		0 \\
		\end{matrix}
		\right] .
	\ee
		The only difference with respect to the consistent case is that $\hat{A}^R$ in (\ref{213}) is block tridiagonal, instead of $A^R$ in (\ref{AR}) which is block bidiagonal. This will influence only the construction  proposed in \cite{duffsisc} of the mutually orthogonal row blocks matrix $\bar{A}$ from  (\ref{3}). Indeed, and for a clear presentation we will consider the following particular case of the matrix $\hat{A}^R$ from (\ref{213})
	\begin{equation}
	\label{00-1}
    \hat{A}^R=
	\left[ 
	\begin{matrix}
          E_{11} & A_{11}  \\
	A_{11}^T &     E_{22}    & A_{21}^T \\
             & A_{21}   &     E_{33}     & A_{22} \\
             &          &  A_{22}^T &     E_{44}   & A_{32}^T \\
             &          &           &  A_{32} &    E_{55}    & A_{33} \\
             &          &           &         &  A_{33}^T &     E_{66}   & A_{43}^T \\
	\end{matrix}
	\right],  
	\end{equation}
and the matrix $A=A^R$ from (\ref{AR})     
   	\begin{equation}\label{AR1}
    A=
    \left[ 
    \begin{matrix}
    A_{11}  \\
    A_{21}      &  A_{22} \\
                &  A_{32} & A_{33} \\
                &         & A_{43}  \\
    \end{matrix}
    \right]   
    \begin{matrix}
    \}\ n_1 \\
    \}\ n_3 \\
    \}\ n_5 \\
    \}\ n_7 \\
    \end{matrix}
    \end{equation}
        \begin{equation*}
    \ \ \ n_2 \quad \ n_4 \quad \  n_6 \
    \end{equation*}
		\textbf{(i) Consistent case.} For the matrix $A=A^R$ in (\ref{AR1}) we define  $\ba = [A ~~\Gamma ]$ from (\ref{3}) by 
 \begin{equation}\label{Abar}
    \ba = 
    \left[ 
    \begin{matrix}
    A_{11}  &         &         &   A_{11}       &         &  \\
    A_{21}  &  A_{22} &         &  -A_{21}       &-A_{22}  &  \\
            &  A_{32} & A_{33}   &               & A_{32}   & A_{33}\\
            &         & A_{43}    &               &   &       -A_{43} \\
    \end{matrix}
    \right]   
    \ee
 Then, the obvious equality holds
 $$
 \Gamma = \left[ 
    \begin{matrix}
    A_{11}  \\
    -A_{21}      &  -A_{22} \\
                &  A_{32} & A_{33} \\
                &         & -A_{43}  \\
                \end{matrix}
                \right] = D ~A = D ~\left[ 
    \begin{matrix}
    A_{11}  \\
    A_{21}      &  A_{22} \\
                &  A_{32} & A_{33} \\
                &         & A_{43}  \\
                \end{matrix}
                \right],
 $$
 with $D$ given by 
 \be
 \label{d}
 D = \left[ \begin{matrix}
 I_{n_1} & 0 & 0 & 0 \\
 0 & -I_{n_3} & 0 & 0 \\
 0 & 0 & I_{n_5} & 0 \\
 0 & 0 & 0 & -I_{n_7} \\
 \end{matrix} \right]
 \ee
 where $I_{n_j}: n_j \times n_j$ are the appropriate unit matrices (see (\ref{AR1})). Therefore, according to (\ref{Abar}) (see also (\ref{3})) we have
\be
\label{xxx}
\bar{A}_i = [ A_i ~~~(-1)^{i+1} A_i],
\ee
which tells us that, for this construction
\be
\label{p12}
rank(\bar{A}_i) = rank(A_i), \forall i=1, \dots, p.
\ee
\textbf{(ii) Inconsistent case.} For the matrix  
   $A = \hat{A}^R$ in (\ref{00-1}) we define  $\ba = [A ~\Gamma ]$ from (\ref{3}) as follows. First of all we group the blocks as
		\be
		\label{220}
	A = 	\left[ \begin{matrix}
 B_{11} & B_{12} & 0 & 0 \\
 0 & B_{22} & B_{23} & 0 \\
 0 & 0 & B_{33} & B_{34} \\
 \end{matrix} \right],
		\ee
		where the new block are defined by 
		$$
		B_{11} = 	\left[ \begin{matrix}
 E_{11} \\
 A^T_{11} \\
 \end{matrix} \right], 
~B_{12} = 	\left[ \begin{matrix}
 A_{11} & 0\\
 E_{22} & A^T_{21}\\
 \end{matrix} \right], 
~B_{22} = 	\left[ \begin{matrix}
 A_{21} & E_{33}\\
 0 & A^T_{22}\\
 \end{matrix} \right],
		$$
		\be
		\label{221}
B_{23} = 	\left[ \begin{matrix}
 A_{22} & 0\\
 E_{44} & A^T_{32} \\
 \end{matrix} \right], 
~B_{33} = 	\left[ \begin{matrix}
 A_{32} & E_{55}\\
 0 & A^T_{33}\\
 \end{matrix} \right], 
~B_{34} = 	\left[ \begin{matrix}
 A_{33} & 0\\
 E_{66} & A^T_{43}\\
 \end{matrix} \right] .
		\ee
Then we construct the matrix  $\ba = [A ~\Gamma ]$ from (\ref{3}) again as (see (\ref{xxx}) and (\ref{Abar})) 
\begin{equation}\label{222}
    \ba = 
    \left[ 
    \begin{matrix}
    B_{11}  &   B_{12}      &   0 & 0       &   B_{12}       &  0       & 0 \\
    0  &  B_{22} &   B_{23}  & 0    &  -B_{22}       &-B_{23}  & 0 \\
     0 & 0        &  B_{33} & B_{34}   &      0         & B_{33}   & B_{34}\\
    \end{matrix}
    \right] .  
    \ee

\end{document}